\documentclass[12pt]{article}
\usepackage[margin=1in]{geometry}

\usepackage{diagbox,color,amsthm,amsmath,amssymb,enumerate,authblk,multicol}
\usepackage{cite}
\usepackage{xcolor}
\usepackage{hyperref}
\usepackage{pgfplots}
\usepackage{float}
\usepackage{algorithm}
\usepackage{algpseudocode}
\pgfplotsset{compat=1.5}
\newtheorem{theorem}{Theorem}[section]
\newtheorem{corollary}{Corollary}[theorem]
\newtheorem{lemma}[theorem]{Lemma}

\newtheorem{conjecture}[theorem]{Conjecture}

\newtheorem*{theorem*}{Theorem}
\newtheorem{example}[theorem]{Example}
\theoremstyle{definition}

\newcommand{\N}{\mathbb{N}}
\newcommand{\Z}{\mathbb{Z}}

\newcommand{\E}{\mathcal{E}}
\newcommand{\inv}{^{-1}}

\usepackage{atbegshi}% http://ctan.org/pkg/atbegshi
\AtBeginDocument{\AtBeginShipoutNext{\AtBeginShipoutDiscard}}

\begin{document}

\title{Rado Numbers and SAT Computations}

% \author{Yuan Chang, Jes\'{u}s A. De Loera, William J. Wesley}
% \email{merchang@ucdavis.edu, deloera@math.ucdavis.edu, wjwesley@math.ucdavis.edu}
\author[1]{Yuan Chang}
\thanks{merchang@ucdavis.edu}
\author[1]{Jes\'{u}s A. De Loera\thanks{deloera@math.ucdavis.edu} }\author[1]{William J. Wesley\thanks{wjwesley@math.ucdavis.edu}}
\affil[1]{Department of Mathematics, University of California, Davis\\1 Shields Ave\\Davis, CA 95616}
\maketitle

\section{Introduction}
%Ramsey theory is the study of unavoidable patterns in large mathematical structures. The question of when a pattern must necessarily appear is central to computational Ramsey theory. The problem of computing the Ramsey number $R(r,s)$, for example, asks for the smallest $n$ such that every graph on $n$ vertices contains either a clique of size $r$ or an independent set of size $s$. 
A well-known result in arithmetic Ramsey theory is \emph{Schur's theorem}, which states that for any $k\ge 1$, there is an integer $n$ such that every $k$-coloring of $[n]:=\{1,2,3,\dots,n\}$ induces a monochromatic solution to $x+y = z$ \cite{Schur1917}. The \emph{Schur number} $S(k)$ is the smallest such $n$ with this property. It is very difficult to compute $S(k)$ and the largest known Schur number is $S(5) = 161$, which was computed in 2018 using SAT solving techniques and created great attention with its massive parallel computation \cite{SchurFive}. In this paper we follow this symbolic approach to investigate natural variations of Schur numbers. 

The \emph{Rado numbers} are a generalization of Schur numbers, and are of great importance in arithmetic Ramsey theory (see \cite{GrahamRothschildSpencer,LandmanRobertson,RadoThesis} and the references below). For a given linear equation $\E$, the \emph{$k$-color Rado number} $R_k(\E)$ is the smallest number $n$ such that every $k$-coloring of $[n]$ induces a monochromatic solution to $\E$, or infinity if there is a $k$-coloring of $\N$ with no monochromatic solution to $\E$.  
Many 2-color Rado numbers for various types of equations have been computed in, for example, \cite{MyersThesis,SchaalZinter,Saracino2colorRado,PythagoreanTriples}. Explicit formulas for the Rado numbers $R_2(a(x-y) = bz)$ and $R_2(a(x+y) = bz)$ are given in \cite{LandmanRobertson} and \cite{HarborthMaasberg}. However, no such formula is known for the general case $R_2(ax+by = cz)$. There are some computations of 3-color Rado numbers scattered throughout the literature \cite{MyersThesis,SchaalFamily3-colorRado,RendallTrumanSchaal}, but Rado numbers with more than two colors have not been studied as often. We present a systematic study of these numbers.

Another interesting family of numbers is the \emph{generalized Schur numbers} $S(m,k) = R_k(x_1+\dots+x_{m-1} = x_m)$. In \cite{BBGeneralizedSchur} it was shown that $S(m,3) \ge m^3-m^2-m-1$, and it was conjectured in \cite{AhmedSchaalGenSchur} and later proved in \cite{BMRS_3ColorSchur} that $S(m,3) = m^3-m^2-m-1$. Myers showed in \cite{MyersThesis} that the numbers $R_k(x-y = (m-2)z)$ give an upper bound for $S(m,k)$, and several more values of $R_k(x-y = (m-2)z)$ were shown to be equal to $S(m,k)$, thus giving exact values for more generalized Schur numbers. Myers went on to make the following conjecture in \cite{MyersThesis}.

\begin{conjecture}[Myers]\label{MyersConjecture}
    $R_3(x-y = (m-2)z) = m^3-m^2-m -1$ for $m\ge 3$.
\end{conjecture}

In this paper we focus on computing Rado numbers for three variable linear homogeneous equations using SAT-based methods described in Section \ref{SATEncoding}.
The main contributions of this paper are exact formulas for several families of 3-color Rado numbers. In particular, we show Conjecture \ref{MyersConjecture} is true. 

\begin{theorem}\label{Theorem3ColorRadoFamilies}
The values of the following Rado numbers are known:
\begin{enumerate}
    \item $R_3(x-y = (m-2)z) = m^3-m^2-m-1$ for $m\ge 3.$
    
    \item $R_3(a(x-y) = (a-1)z) = a^3+(a-1)^2$ for $a\ge 3$.
    
    \item $R_3(a(x-y) = bz) = a^3$ for $b\ge 1$, $a\ge b+2$, $gcd(a,b)=1$.
\end{enumerate}

\end{theorem}
As a corollary, we obtain the result in \cite{BMRS_3ColorSchur}, an exact formula for the 3-color generalized Schur numbers.

\begin{corollary}\label{CorollaryGenSchur}
$S(m,3) = m^3-m^2-m-1$ for $m\ge 1$.
\end{corollary}

We also compute exactly several 3-color and 4-color Rado numbers. 
\begin{theorem}\label{MainRadoNumberComputations}The values of the following Rado numbers are known. 

\begin{enumerate}

\item $R_3(a(x-y) = bz)$ for $1\le a,b \le 15$. 

\item $R_3(a(x+y) = bz)$ for $1\le a,b \le 10$. 

\item $R_3(ax+by = cz)$ for $1\le a,b,c \le 6$. 

% \item $R_3(x-y = az)$ for $1\le a \le 45$.

\item $R_4(x-y = az)$ for $1\le a \le 4$.

\item $R_4(a(x-y) = z)$ for $1\le a \le 5$. 
\end{enumerate}
\end{theorem}

We collect the 3-color Rado number values we computed in Theorem \ref{MainRadoNumberComputations} in Tables \ref{Threecoloraxminusayequalsbz} to \ref{TableR3_6z}. We also give the additional values $R_3(ax+ay = bz)$ for $3\le a \le 6,\ 11\le b \le 20$ as well as our values for $R_4(a(x-y) = bz)$ in the Appendix (Tables \ref{Table_Rado_3Color_additional} and \ref{Table_Rado_4Color_additional}). Underlined entries in these tables correspond to equations whose coefficients are not coprime.

\begin{center}
\begin{table*}[h!] \footnotesize
\caption{3-color Rado numbers $R_3(a(x-y) = bz)$}
\begin{tabular}{c|ccccccccccccccc}\label{Threecoloraxminusayequalsbz}
\diagbox{$b$}{$a$} & 1 & 2 & 3 & 4 & 5 & 6 & 7 & 8 & 9 & 10 & 11 &12 &13 &14 & 15\\
\hline
1 & 14 & 14 & 27 & 64 &125 &216 & 343 & 512 & 729 & 1000 & 1331 & 1728 & 2197 & 2744 & 3375\\

2 & 43 & \underline{14} & 31 & \underline{14} &125 &\underline{27} &343&\underline{64} & 729 & \underline{125} & 1331 & \underline{216} & 2197 & \underline{343} & 3375\\

3 & 94 & 61 & \underline{14} & 73 &125&\underline{14}&343&512 & \underline{27} & 1000 & 1331 & \underline{64} & 2197 & 2744 & \underline{125}\\

4 & 173 & \underline{43} & 109 & \underline{14} &141&\underline{31}&343&\underline{14} & 729 & \underline{125} & 1331 &\underline{27} & 2197 & \underline{343} & 3375\\

5 & 286 & 181 & 186 & 180 &\underline{14} &241 & 343 & 512 & 729 & \underline{14} & 1331 & 1728 & 2197 & 2744 & \underline{27}\\

6 & 439 & \underline{94} & \underline{43}  & \underline{61} &300&\underline{14}& 379 &\underline{73} & \underline{31} & \underline{125} & 1331 & \underline{14} &2197 & \underline{343} & \underline{125}\\

7 & 638 & 428 & 442 & 456 &470&462 &\underline{14}&561 & 729 & 1000 & 1331 & 1728 &2197 & \underline{14} & 3375\\

8 & 889 & \underline{173} & 633 & \underline{43}&665&\underline{109}& 644 &\underline{14} & 793 & \underline{141} & 1331 & \underline{31} &2197 & \underline{343} & 3375\\

9 & 1198 & 856 & \underline{94}& 892 & 910 & \underline{61} & 896 & 896 & \underline{14} & 1081 & 1331 &\underline{73}&2197 & 2744 & \underline{125} \\

10 & 1571 & \underline{286} & 1171 & \underline{181} &\underline{43}& \underline{186} & 1190 &\underline{180} & 1206 &\underline{14}& 1431 &\underline{241}&2197 & \underline{343} &\underline{31}\\

11 & 2014 & 1508 & 1530 & 1552 & 1574 & 1596 & 1618 &1584&1575 &1580 & \underline{14} & 1849 &2197 & 2744 & 3375\\

12 & 2533 & \underline{439} & \underline{173} & \underline{94} & 2005 & \underline{43} & 2053 & \underline{61} & \underline{109} & \underline{300} & 2024 & \underline{14} & 2341& \underline{379}  & \underline{141} \\

13 & 3134 & 2432 & 2458 & 2484 & 2510 & 2536 & 2562 & 2588 & 2574 & 2530 & 2541 & 2544 &\underline{14} & 2913 & 3375  \\

14 & 3823 & \underline{638} & 3039 &\underline{428} & 3095 & \underline{442} & \underline{43} & \underline{456} & 3207 &\underline{470} & 3113 &\underline{462} & 3146 &\underline{14} & 3571\\

15 & 4606 & 3676 &\underline{286} & 3736 &\underline{94} &\underline{181} &3826 & 3856 &\underline{186} & \underline{61} & 3795 & \underline{180} & 3835 & 3836 &\underline{14}
\end{tabular}
\end{table*}
\end{center}

\begin{table*}[h!]
\caption{3-color Rado numbers $R_3(a(x+y) = bz)$}
\begin{center}
\begin{tabular}{c|cccccccccc}\label{3colorTableaxaybz}
    \diagbox{$b$}{$a$} & 1 & 2 & 3 & 4 & 5 & 6 & 7 & 8 & 9 & 10\\
    \hline
     1 & 14 & $\infty$ & $\infty$ &$\infty$ & $\infty$ & $\infty$ & $\infty$ & $\infty$ & $\infty$ & $\infty$\\
     2& 1  &\underline{14} & 243 &$\underline{\infty}$ & $\infty$ & $\underline{\infty}$ & $\infty$ & $\underline{\infty}$ & $\infty$ & $\underline{\infty}$\\
     3& 54 & 54 &\underline{14} & 384 & 2000 & $\underline{\infty}$ & $\infty$ & $\infty$ & $\underline{\infty}$ & $\infty$\\
     4& $\infty$  & \underline{1} & 108 & \underline{14} & 875 & \underline{243} & 4459 & $\underline{\infty}$ & $\infty$ & $\underline{\infty}$\\ 
     5& $\infty$ & 105 & 135 & 180 & \underline{14} & 864 & $3430$ & 3072 & 12393 & $\underline{\infty}$ \\
     6& $\infty$ & \underline{54} &\underline{1}&\underline{54} & 750 & \underline{14} & $3087$ & \underline{384} & \underline{243} & \underline{2000}\\
     7& $\infty$ & 455 & 336 & 308 & 875 & $756$ & \underline{14} & $1536$ & 8748 & 7500\\
     8& $\infty$ & \underline{$\infty$} & 432 &\underline{1} & 1000 & \underline{108} & $2744$ & \underline{14} & 8019 & \underline{875}\\
     9 & $\infty$ & $\infty$ & \underline{54} & 585 & 1125 & \underline{54} & 3087 & 1224 & \underline{14} & 6000\\
     10 & $\infty$ & $\underline{\infty}$ & 1125 & \underline{105} & \underline{1} & \underline{135} & 3430 & \underline{180} & 7290 & \underline{14}
     
\end{tabular}
\end{center}
\end{table*}

% \begin{table*}
% \caption{3-color Rado numbers $R_3(a(x+y) = bz)$}
% \begin{tabular}{c|cccccccccc}\label{3colorTableaxaybz}
%     \diagbox{$b$}{$a$} & 1 & 2 & 3 & 4 & 5 & 6 & 7 & 8 & 9 & 10\\
%     \hline
%      1 & 14 & $\infty$ & $\infty$ &$\infty$ & $\infty$ & $\infty$ & $\infty$ & $\infty$ & $\infty$ & $\infty$\\
%      2& 1  &\textit{14} & 243 &$\infty$ & $\infty$ & $\infty$ & $\infty$ & $\infty$ & $\infty$ & $\infty$\\
%      3& 54 & 54 &\textit{14} & 384 & 2000 & $\infty$ & $\infty$ & $\infty$ & $\infty$ & $\infty$\\
%      4& $\infty$  & \textit{1} & 108 & \textit{14} & 875 & \textit{243} & 4459 & $\infty$ & $\infty$ & $\infty$\\ 
%      5& $\infty$ & 105 & 135 & 180 & \textit{14} & 864 & $3430$ & 3072 & 12393 & $\infty$ \\
%      6& $\infty$ & \textit{54} &\textit{1}&\textit{54} & 750 & \textit{14} & $3087$ & \textit{384} & \textit{243} & \textit{2000}\\
%      7& $\infty$ & 455 & 336 & 308 & 875 & $756$ & \textit{14} & $1536$ & 8748 & 7500\\
%      8& $\infty$ & \textit{$\infty$} & 432 &\textit{1} & 1000 & \textit{108} & $2744$ & \textit{14} & 8019 & \textit{875}\\
%      9 & $\infty$ & $\infty$ & \textit{54} & 585 & 1125 & \textit{54} & 3087 & 1224 & \textit{14} & 6000\\
%      10 & $\infty$ & $\infty$ & 1125 & \textit{105} & \textit{1} & \textit{135} & 3430 & \textit{180} & 7290 & \textit{14}
     
% \end{tabular}
% \end{table*}

%<------------------------------------------>

If the number $R_k(\E)$ is finite for a fixed $k$, we say that the equation $\E$ is \emph{$k$-regular}. If $\E$ is $k$-regular for all $k\ge 1$, we say $\E$ is \emph{regular}. In its simplest version, Rado's classical theorem gives necessary and sufficient conditions for when a linear homogeneous equation is regular \cite{RadoThesis}. 

\begin{theorem}[Rado]\label{RadoRegularityThm}
    A linear equation $\sum_{i=1}^m c_i x_i = 0$ with $c_i \in \Z$ is regular if and only if there exists a nonempty subset of the $c_i$ that sums to zero.  
\end{theorem}

Nonregular equations may have finite Rado numbers for small $k$. The largest $k$ for which an equation $\E$ is $k$-regular is the \emph{degree of regularity} of $\E$, denoted $dor(\E)$. We set $dor(\E):=\infty$ if $\E$ is regular. Rado also proved a theorem that characterized the 2-regular linear homogeneous equations in three or more variables. 

\begin{theorem}[Rado]\label{Rado2RegularityThm}
Let $m\ge 3$ and $a_1,\dots,a_m \in \Z\setminus \{0\}.$ Then the equation $\sum_{i=1}^m a_i x_i = 0$ is 2-regular if and only if there exists $i$ and $j$ such that $a_i >0$ and $a_j <0$. 
\end{theorem} 

For $3 \le k < \infty$, there is no known characterization of the $k$-regular equations. In \cite{AlexeevTsimerman} it was shown that for every $k$, there exists a linear homogeneous equation in $k+1$ variables that has degree of regularity $k$. However, for a fixed number of variables, the question of what degrees of regularity are possible for homogeneous linear equations remains largely unanswered. Rado made the following conjecture about this question in his Ph.D. thesis \cite{RadoThesis}.\\

\begin{table}[H]
\begin{minipage}[c]{0.45\textwidth}
\caption{$R_3(ax+by = z)$}
\begin{tabular}{c|cccccc}
\diagbox{$b$}{$a$} &1 & 2 &3&4&5 & 6 \\
\hline
1 & 14 & 43 & 94 & 173 & 286 & 439 \\
2 &  &$\infty$  & 1093& $\infty$   &  2975 & 4422 \\
3 &  &  & $\infty$ & $\infty$ & $\infty$ & $\infty$\\
4 &  &  &  & $\infty$ & $\infty$ & $\infty$\\
5 & & & & & $\infty$ & $\infty $ \\
6 & & & & & &  $\infty$\\
\end{tabular}
\end{minipage}
\hspace{20pt}
\begin{minipage}[c]{0.45\textwidth}
\caption{$R_3(ax+by = 2z)$}
\begin{tabular}{c|cccccc}
\diagbox{$b$}{$a$} &1 & 2 &3&4&5 & 6 \\
\hline
1 & 1 & 14  & 54  & $\infty$ & 70  & 126 \\
2 &  & \underline{14}  & 61& \underline{43}  & 181 & \underline{94} \\
3 &  &  &243  & $\infty$ & 395 &648 \\
4 &  &  &  & $\underline{\infty}$  & $\infty$  & $\underline{1093}$ \\
5 & & & & & $\infty$ & $\infty$ \\
6 & & & & & & $\underline{\infty}$ \\
\end{tabular}
\end{minipage}
\end{table}
% \newline
% \begin{table*}
% \caption{$R_3(ax+by = 2z)$}
% \begin{center}
% \begin{tabular}{c|cccccc}
% \diagbox{$b$}{$a$} &1 & 2 &3&4&5 & 6 \\
% \hline
% 1 & 1 & 14  & 54  & $\infty$ & 70  & 126 \\
% 2 &  & \underline{14}  & 61& \underline{43}  & 181 & \underline{94} \\
% 3 &  &  &243  & $\infty$ & 395 &648 \\
% 4 &  &  &  & $\underline{\infty}$  & $\infty$  & $\underline{1093}$ \\
% 5 & & & & & $\infty$ & $\infty$ \\
% 6 & & & & & & $\underline{\infty}$ \\
% \end{tabular}
% \end{center}
% \end{table*}
\begin{table}[H]
\begin{minipage}[c]{0.45\textwidth}
\caption{$R_3(ax+by = 3z)$}
\begin{tabular}{c|cccccc}
\diagbox{$b$}{$a$} &1 & 2 &3&4&5 & 6 \\
\hline
1 & 54 & 1 & 27 & 54 & 89 & 195  \\
2 &  & 54 & 31 & $\infty$   & 140  & 108 \\
3 &  &  & \underline{14} & 109  & 186 & \underline{43}\\
4 &  &  &  & 384  & 220 &  $\infty$\\
5 & & & & &  2000& 1074  \\
6 & & & & & & $\underline{\infty}$ \\
\end{tabular}
\end{minipage}
\hspace{20pt}
\begin{minipage}[c]{0.45\textwidth}\caption{$R_3(ax+by = 4z)$}
\begin{tabular}{c|cccccc}
\diagbox{$b$}{$a$} &1 & 2 &3&4&5 & 6 \\
\hline
1 &$\infty$  &$\infty$  & 1  & 64 & 100 & $\infty$   \\
2 &  & \underline{1} & $\infty$ & \underline{14}   &$\infty$  & \underline{54}  \\
3 &  &  & 108 & 73  & 105 & $\infty$\\
4 &  &  &  & \underline{14}  & 180 & \underline{61} \\
5 & & & & & 141 & $\infty$   \\
6 & & & & & & \underline{31} \\
\end{tabular}
\end{minipage}
\end{table}

% \newline
% $R_3(x+4y = 3z) = 54$,
% $R_3(x+5y = 3z) = 89$,
% $R_3(2x+5y = 3z) = 140$,
% $R_3(4x+5y = 3z) = 220$
% $R_3(4x+11y = 3z) = 693$\\
% \begin{table}

% \begin{center}

% \end{center}
% \end{table}

% $R_3(x+2y =5z) = 45$,
% $R_3(x+3y = 5z) = 60$,
% $R_3(3x+4y = 5z) = 100$
% $R_3(4x+2y=5z) \ge 5000$,
% $R_4(2x+2y=5z) \ge 3000$,
% \begin{table}
% \caption{$R_3(ax+by = 5z)$}
% \begin{tabular}{c|cccccc}
% \diagbox{$b$}{$a$} &1 & 2 &3&4&5 & 6 \\
% \hline
% 1 & $\infty$ & 45& 60&1 & 125 &150   \\
% 2 &  &105 &1 &$\infty$ &125 & 70 \\
% 3 &  & & 135&100 &125 &108\\
% 4 &  & & & 180 &141 & $\infty$\\
% 5 &  & & & &$\textit{14}$ & 300  \\
% 6 &  & & & & & 864 \\
% \end{tabular}
% \end{table}
\begin{table}[H]
\begin{minipage}[c]{0.45\textwidth}

\caption{$R_3(ax+by = 5z)$}

\begin{tabular}{c|cccccc}
\diagbox{$b$}{$a$} &1 & 2 &3&4&5 & 6 \\
\hline
1 & $\infty$ & 45& 60&1 & 125 &150   \\
2 &  &105 &1 &$\infty$ &125 & 70 \\
3 &  & & 135&100 &125 &108\\
4 &  & & & 180 &141 & $\infty$\\
5 &  & & & &$\underline{14}$ & 300  \\
6 &  & & & & & 864 \\
\end{tabular}
\end{minipage}
\hspace{20pt}
\begin{minipage}[c]{0.45\textwidth}
\caption{$R_3(ax+by = 6z)$}

\begin{tabular}{c|cccccc}\label{TableR3_6z}
\diagbox{$b$}{$a$} &1 & 2 &3&4&5 & 6 \\
\hline
1 & $\infty$ & 40 &81&$\infty$&1&216 \\
2 &  &\underline{54}&81&\underline{1}&90&\underline{27}\\
3 &  & &\underline{1}&$\infty$&135&\underline{14}\\
4 &  & & &\underline{54}&$\infty$&\underline{31}\\
5 &  & & & &750&241\\
6 &  & & & & &\underline{14}\\
\end{tabular}
\end{minipage}
\end{table}

\begin{conjecture}[Rado's Boundedness Conjecture]
    For all $m\ge1$, there is a number $\Delta(m)$ such that if a linear equation in $m$ variables is $\Delta(m)$-regular, then it is regular. 
\end{conjecture}

In \cite{FoxKleitman} it was shown that Rado's boundedness conjecture is true if it is true for the case of homogeneous equations. Those authors also proved the first nontrivial case of the conjecture by showing that if a linear homogeneous equation in three variables is 24-regular, then it is regular. However, it is not known if 24 is the best possible bound. There are no known examples of a nonregular linear homogeneous equation in three variables that is 4-regular. Moreover, in \cite{FoxKleitman}, several coloring lemmas give more precise bounds on the degree of regularity of 3-variable linear homogeneous equations. We contribute further improvements on their results and compute the degree of regularity of all sufficiently small equations $ax+by = cz$.

\begin{theorem}\label{TheoremDORComputations}
The degree of regularity of the equation $ax+by = cz$ is known for all $1\le a,b,c \le 5$.
\end{theorem}

We also mention a related conjecture of Golowich \cite{GolowichRegularity}.
\begin{conjecture}\label{ConjectureGolowich}
        For each positive integer $k$ there is an integer $m(k)$ such that for any $m \geq m(k)$, any linear homogeneous equation in $m$ variables with nonzero integer coefficients not all of the same sign is $k-$regular.
    \end{conjecture}

We provide counterexamples to the Golowich conjecture.
We show that for all $m,k \ge 3$ there is a linear homogeneous equation in $m$ variables that is not $k$-regular. 
%These equations are counterexamples to Conjecture \ref{ConjectureGolowich}. 

\begin{theorem}\label{TheoremDORUpperBoundAllm}
For all $m,k \ge 3$, there is a linear homogeneous equation $\E$ in $m$ variables that is not $k$-regular. In particular, $$x_1+\dots+x_{m-1}= \lceil(m-1)^\frac{k-1}{k-2}\rceil x_m$$ is not $k$-regular. Thus Conjecture \ref{ConjectureGolowich} is false.
\end{theorem}
As a corollary, for $m\ge 3$ there exists a linear homogeneous equation in $m$ variables with degree of regularity exactly 2. 

\begin{corollary}\label{CorollaryDOR2Allm}
For all $m \ge 3$, there is a linear homogeneous equation $\E$ in $m$ variables with $dor(\E) = 2$. In particular, $$dor(x_1+\dots+x_{m-1}= (m-1)^2x_m) = 2.$$
\end{corollary}

This paper is organized as follows: in Section \ref{SATEncoding} we give background on SAT solving and the computational methods we used to compute Rado numbers. Section \ref{ColoringLemmas} gives several lemmas used to obtain improved bounds on the degree of regularity of certain families of equations. In Section \ref{TwoParamRadoNumbers} we introduce a new SAT method to compute families of Rado numbers and prove our main results. Section \ref{SectionRadoCNFFileGeneration} gives additional details on the computational aspects of this project. The Appendix contains additional tables of Rado numbers and experimental data.

\section{SAT Solving and Encoding}\label{SATEncoding}
In this section we explain how to encode the problem of finding Rado numbers as an instance of SAT and describe additional techniques used to increase performance. The code used for our computations can be found at \cite{Ramsey_Research_Software}. 
\subsection{Background on Satisfiability}
A \emph{literal} is a Boolean variable or its negation. A \emph{clause} is a logical disjunction of literals, and a Boolean formula is in \emph{conjunctive normal form} (CNF) if it is a logical conjunction of clauses. 

The Boolean satisfiability problem (SAT) is the problem of determining whether a given Boolean formula is satisfiable, i.e., there is a true/false assignment to the literals that makes the formula true. Any Boolean formula can be expressed in conjunctive normal form, and most SAT solvers take CNF formulas as input.   

\subsubsection{Encoding of the Problem}
The problem of computing Rado numbers can be encoded as an instance of SAT. Given an equation $\E$ and positive integers $n,k$, we construct a formula $F_n^k(\E)$ that is satisfiable if and only if there exists a $k$-coloring of $[n]$ that does not contain a monochromatic solution to $\E$. Therefore if $F_n^k(\E)$ is satisfiable, then $R_k(\E) > n$, and otherwise $R_k(\E) \le n$. We use the variables $v_{j}^i$ that are assigned the value true if and only if the integer $j$ is colored with color $i$. Following the language used in \cite{SchurFive}, a formula $F_n^k(\E)$ consists of three different types of clauses: \emph{positive, negative,} and \emph{optional}. 
    \begin{itemize}
     
        \item Positive clauses encode that every number $j$ is assigned at least one color, and are of the form $v_j^1 \vee v_j^2 \vee \dots \vee v_j^k$ for $1\le j \le n$.
        \item Negative clauses encode that there are no monochromatic solutions to $\E$. If $(x_1, x_2,\dots, x_m)$ is a solution to $\E$, then its corresponding negative clauses are $\bar{v}_{x_1}^i \vee \dots \vee \bar{v}_{x_m}^i$ for $1\le i \le k$. Every positive integer solution $x$ to $\E$ with $\|x\|_{\infty} \le n$ contributes these $k$ negative clauses to $F_n^k(\E)$. 
        \item Optional clauses encode that every number $j$ is assigned at most one color, and are of the form $\bar{v}_j^{i_1} \vee \bar{v}_j^{i_2}$ for $1\le j \le n$ and $1\le i_1 < i_2 \le k$. These clauses are not strictly necessary since they do not affect the satisfiability of $F_n^k(\E)$, but they ensure that satisfying assignments are in one-to-one correspondence with $k$-colorings of $[n]$ that avoid monochromatic solutions to $\E$.  
    \end{itemize}

% \begin{definition}[Encoding]
%     Let $F_k^n(\E)$ denote the encoding, in conjunctive normal form, of a $k-$coloring of the integers from $1$ to $n$ which avoids monochromatic solutions to the equation $\E$.
% \end{definition}

% If $R^n_k(\E)$ is satisfiable, then we can avoid monochromatic solutions with a $k-$coloring of integers $[1,n]$ and $R_k(\E) > n$. If it's unsatisfiable, then $R_k(\E) \le n$.
\begin{example}
The clauses in the formula $F^3_4(x+y=z)$ are: 

Positive clauses:
\begin{align*}
    (\textcolor{red}{v^1_1} \vee \textcolor{blue}{v^2_1} \vee \textcolor{brown}{v^3_1}) \wedge (\textcolor{red}{v^1_2} \vee \textcolor{blue}{v^2_2} \vee \textcolor{brown}{v^3_2}) \wedge (\textcolor{red}{v^1_3} \vee \textcolor{blue}{v^2_3} \vee \textcolor{brown}{v^3_3}) \wedge (\textcolor{red}{v^1_4} \vee \textcolor{blue}{v^2_4} \vee \textcolor{brown}{v^3_4})
\end{align*}

Negative clauses:
\begin{align*}
&(\textcolor{red}{\overline{v}^1_1} \vee \textcolor{red}{\overline{v}^1_1} \vee \textcolor{red}{\overline{v}^1_2}) \wedge (\textcolor{red}{\overline{v}^1_2} \vee \textcolor{red}{\overline{v}^1_1} \vee \textcolor{red}{\overline{v}^1_3}) \wedge (\textcolor{red}{\overline{v}^1_3} \vee \textcolor{red}{\overline{v}^1_1} \vee \textcolor{red}{\overline{v}^1_4}) \wedge \\
&(\textcolor{red}{\overline{v}^1_1} \vee \textcolor{red}{\overline{v}^1_2} \vee \textcolor{red}{\overline{v}^1_3}) \wedge (\textcolor{red}{\overline{v}^1_2} \vee \textcolor{red}{\overline{v}^1_2} \vee \textcolor{red}{\overline{v}^1_4}) \wedge (\textcolor{red}{\overline{v}^1_1} \vee \textcolor{red}{\overline{v}^1_3} \vee \textcolor{red}{\overline{v}^1_4}) \wedge \\
&(\textcolor{blue}{\overline{v}^2_1} \vee \textcolor{blue}{\overline{v}^2_1} \vee \textcolor{blue}{\overline{v}^2_2}) \wedge (\textcolor{blue}{\overline{v}^2_2} \vee \textcolor{blue}{\overline{v}^2_1} \vee \textcolor{blue}{\overline{v}^2_3}) \wedge (\textcolor{blue}{\overline{v}^2_3} \vee \textcolor{blue}{\overline{v}^2_1} \vee \textcolor{blue}{\overline{v}^2_4}) \wedge \\
&(\textcolor{blue}{\overline{v}^2_1} \vee \textcolor{blue}{\overline{v}^2_2} \vee \textcolor{blue}{\overline{v}^2_3}) \wedge (\textcolor{blue}{\overline{v}^2_2} \vee \textcolor{blue}{\overline{v}^2_2} \vee \textcolor{blue}{\overline{v}^2_4}) \wedge (\textcolor{blue}{\overline{v}^2_1} \vee \textcolor{blue}{\overline{v}^2_3} \vee \textcolor{blue}{\overline{v}^2_4})\wedge \\
&(\textcolor{brown}{\overline{v}^3_1} \vee \textcolor{brown}{\overline{v}^3_1} \vee \textcolor{brown}{\overline{v}^3_2}) \wedge (\textcolor{brown}{\overline{v}^3_2} \vee \textcolor{brown}{\overline{v}^3_1} \vee \textcolor{brown}{\overline{v}^3_3}) \wedge (\textcolor{brown}{\overline{v}^3_3} \vee \textcolor{brown}{\overline{v}^3_1} \vee \textcolor{brown}{\overline{v}^3_4}) \wedge \\
&(\textcolor{brown}{\overline{v}^3_1} \vee \textcolor{brown}{\overline{v}^3_2} \vee \textcolor{brown}{\overline{v}^3_3}) \wedge (\textcolor{brown}{\overline{v}^3_2} \vee \textcolor{brown}{\overline{v}^3_2} \vee \textcolor{brown}{\overline{v}^3_4}) \wedge (\textcolor{brown}{\overline{v}^3_1} \vee \textcolor{brown}{\overline{v}^3_3} \vee \textcolor{brown}{\overline{v}^3_4})
\end{align*}

Optional clauses:
\begin{align*}
    &(\textcolor{red}{\overline{v}^1_1} \vee \textcolor{blue}{\overline{v}^2_1}) \wedge (\textcolor{red}{\overline{v}^1_1} \vee \textcolor{brown}{\overline{v}^3_1}) \wedge (\textcolor{blue}{\overline{v}^2_1} \vee \textcolor{brown}{\overline{v}^3_1}) \wedge
    (\textcolor{red}{\overline{v}^1_2} \vee \textcolor{blue}{\overline{v}^2_2}) \wedge (\textcolor{red}{\overline{v}^1_2} \vee \textcolor{brown}{\overline{v}^3_2}) \wedge (\textcolor{blue}{\overline{v}^2_2} \vee \textcolor{brown}{\overline{v}^3_2}) \wedge \\
    &(\textcolor{red}{\overline{v}^1_3} \vee \textcolor{blue}{\overline{v}^2_3}) \wedge (\textcolor{red}{\overline{v}^1_3} \vee \textcolor{brown}{\overline{v}^3_3}) \wedge (\textcolor{blue}{\overline{v}^2_3} \vee \textcolor{brown}{\overline{v}^3_3}) \wedge (\textcolor{red}{\overline{v}^1_4} \vee \textcolor{blue}{\overline{v}^2_4}) \wedge (\textcolor{red}{\overline{v}^1_4} \vee \textcolor{brown}{\overline{v}^3_4}) \wedge (\textcolor{blue}{\overline{v}^2_4} \vee \textcolor{brown}{\overline{v}^3_4})
\end{align*}
\end{example}

% The github repo has the CNF file for $F_4^3(x+y =z)$. 
If we input $F_4^3(x+y=z)$ into a SAT solver, it will output satisfiable. The $3-$coloring  $\textcolor{red}{1},\textcolor{blue}{2}, \textcolor{brown}{3,4}$, for example, avoids monochromatic solutions. We remark that even though some clauses, such as  ${\textcolor{red}{\overline{v}^1_1} \vee \textcolor{red}{\overline{v}^1_1} \vee \textcolor{red}{\overline{v}^1_2}}$, contain redundant literals, these literals are removed in a preprocessing step. 

We will use this SAT encoding to prove Theorem \ref{MainRadoNumberComputations}. In Section \ref{SectionRadoCNFFileGeneration} we give practical details on this encoding and how to generate formulas efficiently.

\section{Degree of regularity coloring lemmas}\label{ColoringLemmas}

Working towards the goal of computing the degree of regularity and Rado number for as many equations as possible, in this section we collect several results on colorings that avoid monochromatic solutions. These colorings give upper bounds on the degree of regularity of certain equations, and this allows us to avoid doing unnecessary computations. We are especially interested in cases where we can show that the degree of regularity of an equation is at most three. In these cases a computation of a (finite) 3-color Rado number is a proof that the degree of regularity equals three. 

The following result gives two algebraic conditions that guarantee an upper bound on the degree of regularity. A version of the first condition was proved in \cite{FoxKleitman}. 
% \begin{lemma} \label{AlgebraicConditionsTheorem} Let $\E$ be the equation $ax+by = cz$ with $a,b,c >0$ and $a\le b$. Then $\E$ is not $k$-regular if either of the following conditions holds:
% \begin{enumerate}[(i)]
%     \item $a+b \le a^{k-1}/c^{k-2}$, 
%     \item $a+b \le c^{\frac{k-2}{k-1}}$.
% \end{enumerate}
% \end{lemma}

% \begin{proof}
% For the first condition, suppose $a+b \le \frac{a^{k-1}}{c^{k-2}}$ and that $(x,y,z)$ is a solution to $\E$. Let $M = \max(x,y)$, and let $d = \left(\frac{a+b}{c}\right)^{\frac{1}{k-1}}$. Define $\chi(n):= \lceil\log_d n\rceil \pmod k$. Suppose $M \in (d^{i-1},d^i]$. Then $\chi(M) = i \pmod k$. We have $z = \frac{ax+by}c \le d^{k-1}M \le d^{i+k-1}.$ and by assumption we have $d^{k-1} = \frac{a+b}{c} \le \left(\frac{a}{c}\right)^{k-1}$, so $z \ge \frac {aM}c \ge dM > d^i $. Therefore $\chi(z) = \lceil \log_d z\rceil  \in [i+1,i+k-1]$, and $\chi(M) \neq \chi(z)$, so there are no $\chi$-monochromatic solutions to $\E$. 

% For the second condition, suppose $a+b \le c^{\frac{k-2}{k-1}}$ and that $(x,y,z)$ is a solution to $\E$. Again let $M = \max(x,y)$ and let $d = c^{\frac{1}{k-1}}$. Define a $k$-coloring $\chi(n) = \lceil \log_d(n)\rceil \pmod k$. Suppose $M \in (d^{i-1},d^i]$. Then $z = \frac{ax+by}c \le \frac{(a+b)M}{c} \le \frac{d^{k-2} M}{d^{k-1}} \le d^{i-1}.$ Moreover, $z \ge \frac{M}{c} > d^{i-k}$. Therefore $
% \chi(z) = \lceil \log_d z \rceil \in [i-k+1,i-1],$ so $\chi(z) \neq \chi(M)$ and there are no $\chi$-monochromatic solutions to $\E$. 
% \end{proof}

\begin{lemma}\label{LemmaAlgebraicConditions}
Let $\E$ be the equation $a_1x_1+\dots+a_{m-1}x_{m-1} = a_m x_m$ with $a_1\le a_2 \le \dots \le a_{m-1}$ and $a_i >0$ for all $i$. Let $S:= \sum_{i=1}^{m-1} a_i$. Then $\E$ is not $k$-regular if one of the following conditions holds:
\begin{enumerate}[(i)]
    \item $S\le \frac{a_1^{k-1}}{a_m^{k-2}}$, 
     \hspace{40pt} (ii)\hspace{4pt}$S \le a_1^{\frac{1}{k-1}}a_m^{1-\frac{1}{k-1}}$.
\end{enumerate}

\end{lemma}

\begin{proof}
For the first condition, let $d:= \left(\frac{S}{a_m}\right)^{\frac{1}{k-1}}$. Define the coloring $\chi(n) := \lceil \log_d n \rceil \pmod k$. Suppose $(x_1,\dots, x_m)$ is a solution to $\E$, and let $M:=\max\{x_1,\dots,x_{m-1}\}$. Let $i$ be the unique integer such that $M\in (d^{i-1},d^i]$. Then $x_m = \frac{\sum_{i=1}^{m-1}a_ix_i}{a_m} \le d^{k-1}M \le d^{i+k-1}$. By hypothesis we have $d^{k-1} \le \left(\frac{a_1}{a_m}\right)^{k-1}$. Therefore $x_m \ge \frac{a_1M}{a_m} \ge dM > d^{i}$. We have shown that $\chi(x_m)\in [i+1,i+k-1]$, so $\chi(x_m) \neq \chi(M)$ and there are no $\chi$-monochromatic solutions to $\E$.  

% For the second condition, suppose $S \le a_m^{\frac{k-2}{k-1}}$ and that $(x_1,\dots,x_m)$ is a solution to $\E$. Again let $M = \max\{x_1,\dots,x_{m-1}\}$ and let $d = a_m^{\frac{1}{k-1}}$. Define a $k$-coloring $\chi(n) = \lceil \log_d(n)\rceil \pmod k$. Suppose $M \in (d^{i-1},d^i]$. Then $x_m = \frac{\sum_{i=1}^{m-1}a_ix_i}{a_m} \le \frac{SM}{a_m} \le \frac{d^{k-2} M}{d^{k-1}} \le d^{i-1}.$ Moreover, $x_m \ge \frac{M}{a_m} > d^{i-k}$. Therefore $
% \chi(x_m) = \lceil \log_d x_m \rceil \in [i-k+1,i-1],$ so $\chi(x_m) \neq \chi(M)$ and there are no $\chi$-monochromatic solutions to $\E$. 

For the second condition, suppose $S\le a_1^{\frac{1}{k-1}}a_m^{1-\frac{1}{k-1}}$ and that $(x_1,\dots,x_m)$ is a solution to $\E$. Again let $M = \max\{x_1,\dots,x_{m-1}\}$, and let $d = \left(\frac{a_1}{a_m}\right)^{\frac{1}{k-1}}.$ Define a $k$-coloring $\chi(n) = \lceil \log_d(n)\rceil \pmod k$. Suppose $M \in (d^{i+1},d^i]$, so $\chi(M) = i \pmod k$. Note $d<1$ since $a_1 <S <= a_1^{\frac{1}{k-1}}a_m^{1-\frac{1}{k-1}}$ implies $a_1<a_m$. Then $x_m = \frac{\sum_{i=1}^{m-1}a_i x_i}{a_m} \le \frac{SM}{a_m} \le dM\le d^{i+1}$. Moreover, $x_m \ge \frac{a_1M}{a_m} = d^{k-1}M>d^{i+k}$. Therefore $\lceil \log_d x_m \rceil \in [i+k-1,i+1]$, so $\chi(x_m) \neq \chi(M)$ and there are no $\chi$-monochromatic solutions to $\E$.
\end{proof}

 Recall that for any prime $p$, the $p$-adic valuation $v_p(x)$ is the largest integer $n$ such that $p^n$ divides $x$. Many useful colorings come from $p$-adic valuations and studying the divisibility properties of an equation's coefficients. We will freely use the fact that $v_p(xy) = v_p(x)+v_p(y)$ for all integers $x$ and $y$. In \cite{FoxKleitman} the following result was shown:
\begin{lemma}\label{LemmaDistinctPrimes}
Suppose $\E$ is an equation of the form
$ax+by+cz = 0$ with $v_p(a), v_p(b), v_p(c)$ all distinct for some prime $p$. Then $\E$ has degree of regularity at most 3.
\end{lemma}
If the condition in Lemma \ref{LemmaDistinctPrimes} is strengthened to distinct $p$-adic valuations modulo 3, then we obtain an improved bound on the degree of regularity. 
\begin{example}
Let $\E$ denote the equation $x+2y = 4z$. Consider the 3-coloring $\chi(n) = v_2(n) \pmod 3$. If $(x,y,z)$ is a solution to $\E$ and $\chi(x) = \chi(y) = \chi(z)$, then $v_2(x), v_2(2y),$ and $v_2(4z)$ are all distinct since these values are all different modulo 3. Let $\alpha = \min\{v_2(x),v_2(2y),v_2(4z)\}$. Then reducing each side of $\E$ modulo $p^{\alpha + 1}$ gives a contradiction. Therefore $\chi$ induces no monochromatic solutions to $\E$. 
\end{example}
The following lemma generalizes the proof in the example above. 
\begin{lemma}\label{LemmaDistinctPrimesModK}
    Let $\E$ be the equation $\sum_{i=1}^m a_i x_i$. If there is a prime $p$ for which $v_p(a_i) \not \equiv v_p(a_j) \pmod k$ for all $i\neq j$, then $\E$ is not $k$-regular. 
\end{lemma}
\begin{proof}
 Define a $k$-coloring $\chi(n):= v_p(n) \pmod k$. Suppose $(x_1,\dots,x_m)$ is a monochromatic solution to $\E$. Then $v_p(a_i x_i) \neq v_p(a_j x_j)$ for $i\neq j$ since these values are distinct modulo $k$. Let $\alpha = \min_{i=1}^m\{v_p(a_ix_i)\}$. Then $\sum_{i=1}^m a_i x_i \not \equiv 0 \pmod{p^\alpha+1}$, so $\sum_{i=1}^m a_i x_i \neq 0$, a contradiction. 
\end{proof}

The following two results are similar to Lemma 5 and Lemma 6 in \cite{FoxKleitman}. Here we show that under additional  assumptions on $v_p(a+b)$ and $v_p(b+c)$, respectively, it follows that the degrees of regularity of certain equations are at most six, which is stronger than the corresponding best bounds in \cite{FoxKleitman}. We also show that another hypothesis on the order a particular group element further improves this upper bound to four.   

\begin{lemma}\label{LemmaUniquePrimes}
    Let $\E$ denote the equation $ax+by+cz = 0$. If $\E$ is not regular and $0 = v_p(a) = v_p(b) = v_p(a+b) <v_p(c)=:r$, then $dor(\E) < 6$. If additionally the element $-ab\inv$ in the multiplicative group  $G = \Z_{p^r}^\times$ has even order, then $dor(\E) < 4$. 
\end{lemma}

\begin{proof}
 Since $v_p(a)=v_p(b) =0$, let $g:=-ab\inv \in G$. Since $a+b \not\equiv 0 \pmod{p^r}$, it follows that $g$ is not the identity element of $G$. Let $\Gamma$ denote the graph with vertex set $\{1,\dots,p^r-1\}$ and edges $(x,y)$ if $x \equiv gy \pmod{p^r}$ (see Figure \ref{fig:CayleyGraph} for an example). Since $v_p(a+b) = 0$, it follows that $-ab\inv \not \equiv 1 \pmod p$, so $\Gamma$ is loopless. Then $\Gamma$ is a union of disjoint cycles, and each cycle has size $\frac{ord(g)}{p^i}$ for some $i$. If $ord(g)$ is even, then all of the cycles in $\Gamma$ have even length, and so $\Gamma$ is 2-colorable (note the conditions $0 = v_p(a) = v_p(b) = v_p(a+b)$ imply $p\neq 2$). Otherwise, $\Gamma$ is 3-colorable since each vertex has degree at most 2. Let $C_1$ be a proper vertex coloring of $\Gamma$ that uses the fewest number of colors. We will construct a (4- or 6-) coloring $C$ to show that $\E$ is not 4- or 6- regular and conclude $dor(\E)<4$ or $dor(\E) < 6$, respectively. 

Let $q:=p^{2r}$. For all $n \in \N$, write $n = q^\alpha n'$ with $n' \not\equiv 0 \pmod q$. Define the coloring $C_2$ to be $$C_2(n) = 
\begin{cases}
1 & \text{if } n' \not \equiv 0 \pmod{p^r},\\
2 & \text{if } n' \equiv 0 \pmod{p^r}.
\end{cases}
$$

Let $C$ be the product coloring $$C(n) =\begin{cases}
(C_1(n'),1) & \text{if } C_2(n) = 1\\
(C_1(n'/p^r),2)& \text{if } C_2(n) = 2.\\ 
\end{cases}. $$ We claim that $C$ is a coloring with no monochromatic solutions to $\E$. 

Suppose $(x,y,z)$ is a monochromatic solution to $\E$. Write $x = q^\alpha x', y = q^\beta y', z = q^\gamma z'$ with $x',y',z'\nmid q$ and $ax+by+cz = 0$. Without loss of generality, we may assume that at least one of $\alpha, \beta$, and $\gamma$ is zero, and in each case we will show a contradiction. 

Suppose first that $C_2(x) = C_2(y) = C_2(z) = 1.$

Case 1: Suppose $\alpha = 0$. Then if $\beta >0$, then we may reduce $\E$ modulo $p^r$ to obtain a contradiction since $b,c \equiv 0 \pmod{p^r}$, but $ax \not \equiv 0 \pmod{p^r}$. So we may assume $\beta = 0$. Now suppose $ax +by \equiv 0 \pmod{p^r}$. Then $y \equiv gx \pmod {p^r}$. But this is impossible since $x$ and $y$ would have different colors by the coloring $C_1$ (recall that $\Gamma$ is loopless). Therefore $ax+by \not \equiv 0 \pmod{p^r}$, and so $ax+by+cz \not \equiv 0 \pmod{p^r}$, a contradiction. 

Case 2: The case $\beta = 0$ is similar to the case $\alpha = 0$.

Case 3: Suppose $\gamma = 0$ and $\alpha, \beta >0$. Then $ax+by \equiv 0 \pmod q$, but $cz \not \equiv 0 \pmod {q}$, and so $ax+by +cz \not \equiv 0 \pmod q, $ a contradiction. Therefore $\alpha = 0$ or $\beta = 0$, and the proof follows from one of the previous cases. 

If $C_2(x) = C_2(y) = C_2(z) = 2$, then in all cases we may divide $x',y',$ and $z'$ by $p^r$, and the proof follows similarly. 

\end{proof}

\begin{lemma}\label{LemmaTwoPrimes}
Let $\E$ denote the equation $ax+by+cz = 0$. If $\E$ is not regular and $0 = v_p(a) <v_p(b)= v_p(c)= v_p(b+c) =: r$ for some prime $p$, then $dor(\E) < 6$. Write $b = p^r b'$ and $c = p^r c'$. If additionally the element $g:=-b'c'^{-1}$ in the multiplicative group  $G = \Z_{p^s}^\times$ has even order, then $dor(\E) < 4$. 
\end{lemma}

\begin{proof}
 Since $v_p(b+c) = r$ and $v_p(b')=v_p(c') = 0$, $g \in G$ and $g$ is not the identity element of $G$. Let $\Gamma$ denote the graph with vertex set $\{1,\dots,p^r-1\}$ and edges $(x,y)$ if $x \equiv gy \pmod{p^r}$. Then $\Gamma$ is a union of disjoint cycles, and each cycle has size $\frac{ord(g)}{p^i}$ for some $i$. Note that $\Gamma$ is not loopless since $v_p(b+c) = r$. If $g$ has even order, then $\Gamma$ is 2-colorable, and otherwise $\Gamma$ is 3-colorable. Let $C_1$ be a proper vertex coloring of $\Gamma$ that uses the fewest number of colors. We will construct a (4- or 6-) coloring $C$ to show that $\E$ is not 4- or 6- regular and conclude $dor(\E)<4$ or $dor(\E) < 6$, respectively. 
 
 Let $q:=p^{2r}$, and for all $n\in \N$, write $n = q^\alpha n'$ with $n' \not \equiv 0 \pmod{p^{2r}}$. Define the coloring $C_2$ to be $$C_2(n) = \begin{cases}
 1 & \text{if }n' \not \equiv 0 \pmod{p^r}, \\
 2 & \text{if }n' \equiv 0 \pmod {p^r}. 
 \end{cases}$$
 
Let $C$ be the product coloring $$C(n) =\begin{cases}
(C_1(n'),1) & \text{if } C_2(n) = 1\\
(C_1(n'/p^r),2)& \text{if } C_2(n) = 2.\\ 
\end{cases}.$$ Suppose $(x,y,z)$ is a monochromatic solution to $\E$ with respect to $C$. Write $x = q^\alpha x', y = q^\beta y', z = q^\gamma z'$ with $x',y',z'\nmid q$ and $ax+by+cz = 0$. Without loss of generality, we may assume that at least one of $\alpha, \beta$, and $\gamma$ is zero, and in each case we will show a contradiction. 
 
 Suppose $C_2(x') = C_2(y') = C_2(z') = 1.$
 Case 1: Suppose $\alpha = 0$. Then $ax+by+cz \not \equiv 0 \pmod{p^r}$.
 Case 2: Suppose $\alpha >0$ and $\beta = 0$. Then if $\gamma >0$, then $ax+by+cz \not \equiv 0 \pmod q$, so assume $\gamma = 0$. If $by+cz \equiv 0 \pmod q$, then $z =z' \equiv -b' c'^{-1}y' \pmod{p^r}$. But by coloring $C_1$ it follows that $z$ and $y$ have different colors, a contradiction. The case $\gamma = 0$ is similar. 
 
 Suppose $C_2(x') = C_2(y') = C_2(z') = 2$. If $\alpha = 0$, then $ax+by+cz \not \equiv 0 \pmod {q}$. The cases $\beta = 0, \gamma = 0$ are similar to above. 
 \end{proof}

\begin{figure}
    \centering
    \includegraphics[scale=0.3]{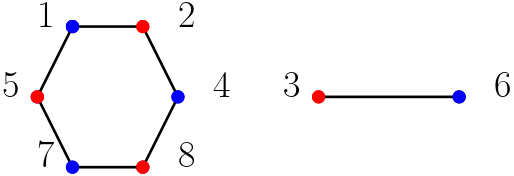}
    \caption{2-coloring of graph with vertex set [8] and edges $(x,y)$ if $x \equiv 2y \pmod 9$. }
    \label{fig:CayleyGraph}
\end{figure}

\section{Two parameter 3-color Rado numbers}\label{TwoParamRadoNumbers}
In \cite{LandmanRobertson} the formulas for the 2-color Rado numbers $R_2(a(x-y) = bz)$ and $R_2(a(x+y) = bz)$, $a,b\ge 0$ are given. However, a formula for the 2-color Rado numbers $R_2(ax+by = cz)$ is unknown. Here we give bounds on some 3-color Rado numbers of the form $R_3(a(x-y) = bz)$ and $R_3(a(x+y) = bz)$ and prove our main results. 
\subsection{Rado Numbers \texorpdfstring{$R_3(a(x-y) = bz)$}{Lg}} 

By Rado's theorem, the equation $a(x-y) = bz$ is regular. Table \ref{Threecoloraxminusayequalsbz} gives values of the 3-color Rado numbers $R_3(a(x-y) = bz)$ for $1\le a,b\le15$.

% We also computed the following values which are not included in the table:
% \begin{proposition}
%     $R_3(a(x-y) = z) = a^3$ for $3\le a\le 25$, $R_3(x-y = bz) = (b+2)^3-(b+2)^2-(b+2)-1$ for $1\le b\le 45$. 
% \end{proposition}

The following lemma gives a simple lower bound on the Rado numbers $R_k(a(x-y) = bz)$.
\begin{lemma}\label{a-adic coloring lower bound}
    Suppose $a,b \ge 1$ and $\gcd(a,b) = 1$. Then $R_k(a(x-y) = bz) \ge a^k$. 
\end{lemma}
\begin{proof}
Let $v_a(n)$ denote the highest power of $a$ that divides $n$. Then $v_a:[1,a^k-1]\to \{0,1,\dots,k-1\}$ defines a $k-$coloring of $[1,a^k-1]$ that has no monochromatic solutions of $a(x-y) = bz$. To see this, suppose $(x,y,z)$ is a monochromatic solution in color $c$. If $x \le y,$ then there is no $z \in [1,a^k-1]$ that satisfies $a(x-y) = bz$, so suppose $x>y$. Then $x = a^c x'$ and $y = a^c y'$, where $a \nmid x',y'$. Since $\gcd(a,b) = 1$, $v_a(z) = v_a(bz) = v_a(a(x-y)) = v_a(a^{c+1}(x'-y'))\ge c+1$, a contradiction. 
\end{proof}

% We were unable to prove this upper bound in general, but we obtained the following lemma. 
% \begin{lemma}
% Let $b\ge 3$, and let $\mathcal{E}$ denote the equation $b(x-y) = z$, then any $3-$coloring $\chi$ of $[b^3]$ that avoids monochromatic solutions to $\mathcal{E}$ must satisfy $\chi(b) \neq \chi(b^2)$ and $\chi(b) \neq \chi(b^2-b+1)$
% \end{lemma}
% \begin{proof}
% {\color{red}Proof is in OneNote. MANY details and tedious; want to make sure it's worth writing up before doing so.}
% {\color{blue} 10/24 Update: there is an easier way to prove lemmas like this; will add details}
% \end{proof}

For the case $b = a-1$, we have an improved lower bound on $R_3(a(x-y) = (a-1)z).$
\begin{lemma}\label{Lemma3colorbEqualsaMinus1LowerBound}
$R_3(a(x-y) = (a-1)z) \ge a^3+(a-1)^2$.
\end{lemma}
\begin{proof}
We will construct a coloring of $[1,a^3+(a-1)^2-1]$ that induces no monochromatic solutions to $a(x-y) = (a-1) z$. 
 Define $$\chi(i):= \begin{cases} 
  0 &v_a(i) = 2 \text{ or } (v_a(i) = 0 \text{ and } (i<a^2-a \text{ or } i>a^3-a)), \\
  1 & v_a(i) = 1,\\
  2 & \text{otherwise.}
 \end{cases}$$
 
 Let $(x,y,z)$ be a positive integer solution to $a(x-y) = (a-1)z$.
 Suppose $\chi(x) = \chi(y) = 0$. If $v_a(x) = v_a(y) \ge 2$, then since $a$ and $a-1$ are relatively prime, $v_a(z) = v_a((a-1)z) = v_a(a(x-y)) \ge 3$, and $\chi(z) \neq 0$. If $v_a(x) = 2$ and $v_a(y) = 0$, then $v_a(z) = v_a(a(x-y)) = 1$, so $\chi(z) = 1$. The case $v_a(x) = 0$ and $v_a(y) = 2$ is similar. Note that $x>y$ since $z$ must be positive. If $v_a(x) = v_a(y) = 0$, $x\ge a^3-a$ and $y \le a^2-a$, then $a(x-y) \ge a(a^3-a^2) > (a-1)(a^3-a)$. Then $z \in [a^3-a,a^3+(a-1)^2-1]$. Since $(a-1) z = a(x-y)$, it follows that $v_a(z) \ge 1$, so $v_a(z)\neq 0$. But the only value  $ z \in [a^3-a,a^3+(a-1)^2-1]$ with $v_a(z) \ge 2$ is $a^3$, so $\chi(z) \neq 0$. Now if $v_a(x) = v_a(y) = 0$ and $x,y \in [1,a^2-a)$ or $x,y \in (a^3-a, a^3+(a-1)^2-1]$, then $x-y < a^2-a$, so $(a-1) z = a(x-y) < a^3-a^2 $. Then $v_a(z) \in \{1,2\}.$ If $v_a(z) = 1$, then $\chi(z) = 1 \neq 0$. If $v_a(z) = 2$, then $z \ge a^2$ and $(a-1)z \ge a^3-a^2$, a contradiction.
 
 Now suppose $\chi(x) = \chi(y) = 1$. Then $v_a(z) = v_a(a(x-y)) \ge 2$, so $\chi(z) \neq 1$. 
 
 If $\chi(x) = \chi(y) = 2$, then either $v_a(x) = v_a(y) = 0$, or $x = a^3$. In the former case we have $v_a(z) = v_a(a(x-y)) \ge 1$. We have $z \neq a^3$ since this implies $x-y = a^3-a^2$, but this is impossible since $x,y \in [1,a^3-(a-1)^2-1]$, and it follows that $\chi(z) \neq 2$. If $x = a^3$, then $y \neq a^3$ and $v_a(y) = 0$. Therefore $v_a(z) = v_a(a(x-y)) = 1$, so $\chi(z) \neq 2$. 
 
 Therefore there are no monochromatic solutions to $a(x-y) = (a-1)z$. 
\end{proof}
% We conjecture that the bound given in Lemma \ref{3colorbEqualsaMinus1LowerBound} is tight. 
% \begin{conjecture}
% $R_3(a(x-y) = (a-1)z) = a^3+(a-1)^2$ for $a\ge 3$.
% \end{conjecture}
% {\color{red} Holds for $a\le 20$; move to intro if can prove}

% In \cite{MyersThesis} Myers made the following conjecture. 

% \begin{conjecture}
%     $R_3(x-y = bz) = (b+2)^3 -(b+2)^2 -(b+2) -1$
% \end{conjecture}
% This conjecture is a stronger version of the existing conjecture that the generalized Schur number $R_3(x_1 + \dots + x_{m-1} = x_m) = m^3 -m^2 -m -1$. Myers computed the values $R_3(x-y = bz)$ for $b = 5,6,7,8$. We computed more evidence for the conjecture as it holds for $b\le 44$. 

\subsection{3-regularity of \texorpdfstring{$a(x+y) = bz$}{Lg}}
In his thesis \cite{RadoThesis}, Rado proved the following result on the family of equations $a(x+y) = bz$. 
\begin{theorem}[Rado] \label{RadoaxaybzThm}
If $a/b \neq 2^k$ for all $k\in \Z$, then $dor(a(x+y) = bz) \le 3$. 
\end{theorem}

The following lemma strengthens Rado's result to include the case when $a/b = 2^k$ for some integer $k$. 

\begin{lemma}
$R_3(x+y= bz) = \infty$ for $b\ge 4$, and $R_3(a(x+y) = z) = \infty$ for $a \ge 2$. Moreover, $dor(a(x+y) =bz) \le 3$ unless $a = 1, b =2$.
\end{lemma}

\begin{proof}
 The results follow immediately from Lemma \ref{LemmaAlgebraicConditions} and Theorem \ref{RadoaxaybzThm}. 
\end{proof}

% Not sure if the theorem below is helpful or not.
% \begin{theorem}[Nathan Johns]
%     For any linear homogeneous equation $S$ in $n$ variables, if it can be expressed as 
%     $$\sum^m_1 a_ix_i - \sum^{n-m}_1 b_ix_i = 0$$
%     where all of $a_1, \dots, a_m$ and $b_1, \dots, b_{n-m}$ are positive integers, such that either $\sum a_i | \sum b_i$, then $S$ is $(n-1)$ regular.
% \end{theorem}

% Jack: rows 5-8 
% Chang: columns 5-8 

Table \ref{3colorTableaxaybz} gives the 3-color Rado numbers $R_3(a(x+y) = bz)$ for $1\le a,b\le10.$ We also give the values of some additional Rado numbers for the equations $a(x+y) = bz$ with $b > 10$ in the Appendix.

\subsection{Proofs of Main Results}

In this section we prove Theorem \ref{Theorem3ColorRadoFamilies} using an encoding of the Rado number problem similar to that in Section \ref{SATEncoding}. The key difference is that in this new encoding, indices of variables are indexed by symbolic polynomial expressions rather than fixed integers. 

Let $\E$ be a linear equation in $m$ variables, and let $S$ be a set of polynomials. Let $C \subseteq S^m$ be a set of solutions to $\E$. The variable $v_s^i$ is assigned the value true if and only if the expression $s\in S$ is assigned color $i$. Positive and optional clauses are constructed similarly to the method in Section \ref{SATEncoding}. The negative clauses are constructed from the solutions in $C$. For example, if $S = \{ia : 1\le i \le 7$\}, and $\E$ is the equation $x-y = 5z$, then $(x,y,z) = (7a,2a,a)$ is a solution. If $(7a,2a,a) \in C$, then we add the negative clause $\bar{v}_{7a}^1 \vee \bar{v}_{2a}^1 \vee \bar{v}_{a}^1$ to our formula. The following lemma formalizes this procedure and describes how to use these formulas to compute Rado numbers for families of equations. 

\begin{lemma}\label{ParametrizedFormulaLemma}
Let $\Sigma = \{\sigma_1,\dots,\sigma_\ell\}$ be a finite alphabet of parameters. Let $\E$ be a linear equation in the variables $x_1,\dots,x_m$ with coefficients in $\Sigma$. Let $S$ be a set of expressions over $\Sigma$, and let $C \subset S^m$ be a set of solutions to $\E$. We define $F_{k,S,C}(\E)$ to be the corresponding formula for the $k$-color Rado number generated by the clauses from $C$ as follows.
\begin{align*}
    F_{k,S,C}(\E) &:=  Pos_{k,S}\wedge Neg_{k,C} \wedge Opt_{k,S}, \text{where} \\
    Pos_{k,S} &:=\bigwedge_{s\in S} \left(\bigvee_{i=1}^k  v_{s}^i \right),  \\
    Neg_{k,C} &:= \bigwedge_{(s_1,\dots,s_m) \in C} \bigwedge_{i=1}^k \bigvee_{j=1}^m \bar{v}_{s_j}^i,\\ Opt_{k,S} &:= \bigwedge_{s\in S}\bigwedge_{1\le i_1< i_2 \le k} (\bar{v}_{s}^{i_1} \vee \bar{v}_{s}^{i_2}).
\end{align*}

Let $A \subset \Z^\ell$.  If $1\le s(a) = s(a_1,\dots,a_\ell) \le f(a)$ for all $s\in S$, $a\in A$ and $F_{k,S}(\E)$ is unsatisfiable, then $R_k(\E) \le f(a)$ for all $a\in A$. 
\end{lemma}

In other words, if substituting values for parameters in a formula $F$ always gives a valid formula, i.e., one whose variables are bounded between 1 and $n$, then the unsatisfiability of $F$ gives an upper bound on the Rado numbers for a family of equations. For each equation $\E$ in the family, $F$ is an unsatisfiable subformula in the corresponding Rado number formula for $\E$.

We are able to prove Conjecture \ref{MyersConjecture} using Lemma \ref{ParametrizedFormulaLemma}. 

\begin{proof}[Proof of Theorem \ref{Theorem3ColorRadoFamilies}]
For the Rado numbers $R_3(x-y = (m-2)z)$, in Lemma \ref{ParametrizedFormulaLemma} let $\Sigma = \{m\}$, and let $\E$ be the equation $x-y = (m-2)z$. The set $S$ is a family of 685 polynomials and the set $C$ contains 9468 solutions to $\E$. It is straightforward to show that $1\le s(m)\le m^3-m^2-m-1$ for all $m\ge 3$ and all $s\in S$. The formula $F_{3,S}(\E)$ was shown to be unsatisfiable in 0.03 seconds by {\scshape{Satch}}, proving $R_3(x-y = (m-2) z) \le m^3-m^2-m-1$ for $m\ge 3$. By results in \cite{BBGeneralizedSchur} and \cite{MyersThesis}, we have $R_3(\E) \ge S(m,3) \ge m^3-m^2-m-1$ for $m\ge 3$, and so $R_3(x-y = (m-2) z) = m^3-m^2-m-1$ for $m\ge 3$.

% \textcolor{red}{1365 polynomials, 20811 negative clauses; need to fix this}

For the Rado numbers $R_3(a(x-y) = (a-1)z)$, let $\Sigma = \{a\}$, and let $\E$ denote the equation $a(x-y) = (a-1)z$. We constructed a set $S$ of 1365 polynomials and a set $C$ of 20811 solutions to $\E$. It is again straightforward to show that $1\le s(a) \le a^3+(a-1)^2$ for all $a\ge 16$ and all $s\in S$. The formula $F_{3,S,C}(\E)$ was shown to be unsatisfiable in 0.05 seconds by {\scshape{Satch}}, proving $R_3(\E) \le a^3+(a-1)^2$ for $a\ge 16$. By Lemma \ref{Lemma3colorbEqualsaMinus1LowerBound} and Theorem \ref{MainRadoNumberComputations}, it follows that $R_3(\E) = a^3+(a-1)^2$ for $a\ge 3$. 

For the Rado numbers $R_3(a(x-y) = bz)$, let $\Sigma = \{a,b\}$, and let $\E$ denote the equation $a(x-y) = bz$. We constructed a set $S$ of 40645 polynomials and a set $C$ of 490897 solutions to $\E$. All polynomials $p(a,b)$ were verified to satisfy $1\le p(a,b) \le a^3$ for all integers $a,b$ satisfying $a\ge 16, b\ge1$, and $a\ge b+2$ using the software {\scshape GloptiPoly 3} \cite{GloptiPoly3}. Some additional valid inequalities were added to the region specified in {\scshape GloptiPoly 3} using elementary calculus techniques. The formula $F_{3,S,C}$ was shown to be unsatisfiable in 1.72 seconds by {\scshape{Satch}}, proving $R_3(a(x-y) = bz) \le a^3$ for all $(a,b)$ satisfying $a\ge 16$, $b\ge 1$, and $a\ge b+2$. By Theorem \ref{MainRadoNumberComputations} and Lemma \ref{a-adic coloring lower bound}, $R_3(a(x-y) = bz) = a^3$ for $a\ge 3, b\ge 1, a\ge b+2$ with $\gcd(a,b) = 1$. 

For each of these formulas, the sets $S$ and $C$ were constructed using a heuristic search procedure. We give details of this procedure in the Appendix.

\end{proof}
The results in Theorem \ref{MainRadoNumberComputations} were proven by computer.

\begin{proof}[Proof of Theorem \ref{MainRadoNumberComputations}] For each finite number $R_k(ax+by = cz)$, we produced a $k$-coloring of $[R_k(ax+by=cz) -1]$ that contained no monochromatic solutions to $ax+by = cz$ and verified using a SAT solver that the formula $F_{R_k(ax+by = cz)}^k(ax+by = cz)$ from the encoding in Section \ref{SATEncoding} is unsatisfiable. For the remaining cases we concluded $R_3(ax+by = cz) = \infty$ using Lemma \ref{LemmaAlgebraicConditions}.
\end{proof}
We are now able to prove Theorem \ref{TheoremDORComputations}. 

\begin{proof}[Proof of Theorem \ref{TheoremDORComputations}]
 Let $\E$ denote the equation $ax+by = cz$. For each triple $(a,b,c)$ that satisfies $1\le a,b,c\le 5$ and $a\le b$, we performed the following calculations. If a nonempty subset of $\{a,b,-c\}$ sums to zero, then $dor(\E) = \infty$ by Theorem \ref{RadoRegularityThm}. If $v_p(a), v_p(b)$, and $v_p(c)$ are all distinct modulo 3 for some prime $p$, or if one of the inequalities $a+b\le \frac{a^2}{c}$ or $a+b \le \sqrt{ac}$ holds, then $dor(\E) = 2$  by Lemma \ref{LemmaDistinctPrimesModK}, Lemma \ref{LemmaAlgebraicConditions}, and Theorem \ref{Rado2RegularityThm}. 

In all other cases $dor(\E) \le 3$ by either Theorem \ref{RadoaxaybzThm}, Lemma \ref{LemmaAlgebraicConditions}, Lemma \ref{LemmaUniquePrimes}, or Lemma \ref{LemmaTwoPrimes}. The computation of a finite Rado number in Theorem \ref{MainRadoNumberComputations} gives $dor(\E) = 3$.  
\end{proof}

We now prove Theorem \ref{TheoremDORUpperBoundAllm} and Corollary \ref{CorollaryDOR2Allm}

\begin{proof}[Proof of Theorem \ref{TheoremDORUpperBoundAllm} and Corollary \ref{CorollaryDOR2Allm}]
The proof of Theorem \ref{TheoremDORUpperBoundAllm} is immediate from Lemma \ref{LemmaAlgebraicConditions} condition $(ii)$ since $S = m-1 \le \lceil (m-1)^{\frac{k-1}{k-2}}\rceil^{\frac{k-2}{k-1}}.$ Corollary \ref{CorollaryDOR2Allm} follows from Theorem \ref{Rado2RegularityThm} and setting $k = 3$ in Theorem \ref{TheoremDORUpperBoundAllm}.
\end{proof}

\section{Rado CNF file generation}\label{SectionRadoCNFFileGeneration}

In this section we discuss some of the computational details from our calculations. The main workflow when computing a Rado number for a given linear equation $\E$ is the following:
\begin{itemize}
    \item Generate the CNF file encoding $F_n^k(\E)$.
    \item Apply symmetry breaking preprocessing.
    \item Determine the satisfiability of $F_n^k(\E)$ with SAT solvers.
    \item Adjust $n$ to find the smallest $n$ for which $F_n^k(\E)$ is unsatisfiable. 
\end{itemize}

In the following subsections, we explain how to achieve each step of the computation procedure. 
\subsection{Generating CNF Files}

Before we compute a given Rado number with a SAT solver, we must write the formula $F_n^k(\E)$ to a file in DIMACS format (see \cite{SATHandbook_1stEd}, Chapter 2). For many of our Rado number calculations, this step took far longer than the SAT solving process. 
The paper \cite{SchurFive} uses divide-and-conquer and several CPU years to solve a single difficult SAT formula in five colors; in contrast, we solve many easier problems with only three colors.
Generating negative clauses is the most difficult step as it involves enumerating the positive integer solutions to $\E$.

\subsubsection{Generating all solutions}

% In order to generate all the negative clauses for $F_n^k(\E)$, we need to find all possible solutions to $\E$ within the given bound of $[1,n]$. A simple way to do this is to loop from $1$ to $n$ for all variables in the equation and check if the given vector is a solution to $\E$, and if it does, generate the appropriate negative clauses and add them to the CNF file. 

% However, this method quickly becomes inefficient as $n$ increases. The next thing we tried was to use the integer program solver SCIP to enumerate all integer solutions of $\E$. Using the output of SCIP, it is easy to generate the negative clauses.

 For efficient solution generation to homogeneous linear equations, we used the built-in function \texttt{isolve} in {\scshape Maple} to parameterize the solutions. We then used {\scshape {SymPy}} to parse the output of {\scshape{Maple}} and generate the solutions with values in $[1,n]$. 
\begin{example}

If we want to generate all integer solutions in the interval $[1,1000]$ for the equation $43x - 5y = 13z$, we can feed $43x - 5y = 13z$ into {\scshape Maple}'s \texttt{isolve} function, which gives the output
$$\{x=i,y=6i-13j, z=i+5j\}.$$
Since we want all integer solutions within $[1,1000]$, we can loop $i$ from $1$ to $1000$ and manipulate the inequality
$$1 \le y=6i-13j \le 1000$$
into an inner loop where $j$ is looped from $\lceil \frac{1000-6i}{-13} \rceil$ to $\lfloor \frac{1-6i}{-13} \rfloor$.
For $z$, we can simply check whether $1 \le i+5j \le 1000$ is satisfied or not inside the loops to determine if $(x,y,z)$ is a  solution. 

\end{example}
% For example, If we want to generate all integer solutions in the range $[1,1000]$ for the equation $43x - 5y = 13z$, we can feed $43x - 5y = 13z$ into \texttt{Maple}'s \texttt{isolve} function, which gives the output
% $$\{x=i,y=6i-13j, z=i+5j\}.$$
% Since we want all integer solutions within $[1,1000]$, we can loop $i$ from $1$ to $1000$ and manipulate the inequality
% $$1 \le y=6i-13j \le 1000$$
% into an inner loop where $j$ is looped from $\lceil \frac{1000-6i}{-13} \rceil$ to $\lfloor \frac{1-6i}{-13} \rfloor$.
% As for $z$, we can simply check if $1 \le i+5j \le 1000$ is satisfied or not inside the loops to determine if this is a solution. We used the \texttt{Sympy} package to parse through \texttt{Maple}'s output in \texttt{Python}.

\subsubsection{Writing Clauses to File}

Some of the formulas in our computations contain millions of clauses, and writing these clauses to a file is a time-consuming part of CNF file generation. For example, the CNF file which encodes $F_{16397}^3(5x+5y=19z)$ contains more than 20 million clauses, of which only $65588$ are positive or optional.

% Another way to increase the performance of the CNF generation is with some hardware optimization. Due to the way the code is written, we can optimize the speed of writes at the small expense of making the code a bit less general. One thing to speed up this process is to reduce the amount of \textit{expensive} operations, such as for loops and conditional statements, in the code. 

The algorithm to generate all the positive and optional clauses is done though {\scshape {Python}}. Since the parameterization of the equation $\E$ is also passed to {\scshape{Python}}, we also used {\scshape{Python}} to generate the CNF files. We were able to improve CNF file generation speeds for 3-color Rado numbers by implementing loop unrolling. 
\subsection{Symmetry Breaking}

Symmetry breaking is a SAT solving technique that can lead to drastic speedups by preventing the solver from looking in isomorphic areas of the search space. In our case we want to prevent the solver from searching for different permutations of the same coloring. 
% For example, there are six ways to permute the colors in the coloring $(\textcolor{red}{1},\textcolor{blue}{2}, \textcolor{brown}{3,4})$: 

We can break this symmetry in the formula $F_4^3(x+y =z)$, for example, by adding the clauses $\textcolor{red}{(v^1_1)}$ and $\textcolor{blue}{(v^2_2)}$. These clauses force number $1$ to be red and number $2$ to be blue. In general, if $\E$ is a linear homogeneous equation in three variables and we have a solution $(x,y,z)$ where two of $x,y,$ and $z$ are equal to each other, then we can add clauses that force the two equal variables to be the first color and the remaining variable to be the second color. For Rado numbers $R_k(\E)$ with $k>3$, we can also add clauses that break the symmetries on the other colors (see \cite{SchurFive}).

Generating a larger set of symmetry breaking clauses with more sophisticated preprocessing is more difficult and requires far more time than normal file generation. We included only a simple preprocessing step in our solving process. The benefit of this preprocessing becomes more apparent when the number of integers to color or the number of colors grows. Without symmetry breaking, {\scshape Satch} takes nearly 15 minutes to determine that $F^4_{45}(x+y=z)$ is unsatisfiable, but only a few seconds after adding symmetry breaking clauses.

% There is a lot more work to be done with preprocessing in the future, such as reducing the size of clauses or eliminating some clauses altogether due to redundancies. 
% \begin{example}
% In the negative clauses of $R^3_4(x+y=z)$ case, we can have reductions like:
% $$(\textcolor{red}{\overline{v}^1_1} \vee \textcolor{red}{\overline{v}^1_1} \vee \textcolor{red}{\overline{v}^1_2}) \to (\textcolor{red}{\overline{v}^1_1} \vee  \textcolor{red}{\overline{v}^1_2})$$
% $$(\textcolor{red}{\overline{v}^1_1} \vee \textcolor{red}{\overline{v}^1_2} \vee \textcolor{red}{\overline{v}^1_3}) \wedge (\textcolor{red}{\overline{v}^1_1} \vee  \textcolor{red}{\overline{v}^1_2}) \to (\textcolor{red}{\overline{v}^1_1} \vee  \textcolor{red}{\overline{v}^1_2})$$

% Since it is redundant to have the same literal repeat itself inside the same clause, and if a clause is a subset of another clause, the second clause is subsumed by the first, making it redundant as well.
% \end{example}

\subsection{SAT Solvers}

Most of the SAT solving computations were done with the solver {\scshape Satch v0.4.17}, developed by Biere \cite{Satch2021}. We initially used {\scshape Satch} because it is remarkably fast at proving upper bounds for many 3-color Rado numbers and relatively easy to use. For example, {\scshape Satch} is able to prove the upper bounds for the values in the first column and last row of Table \ref{Threecoloraxminusayequalsbz} in under 10 seconds for each equation. Later, as we moved towards larger CNF files and more colors, {\scshape Satch} started to struggle.

In order to take advantage of the computation hardware that we had, we also experimented with  the multithreaded SAT solver {\scshape Glucose} \cite{Glucose}. In general, {\scshape Satch} performed better on smaller instances, but {\scshape Glucose} solved larger instances up to two times as quickly.

% \textcolor{red}{Streamline this part (satch for smaller, glucose for larger)}
% {\scshape Glucose} is a multithreaded SAT solver that can be faster than {\scshape Satch} with appropriate settings.
% {\scshape Glucose} has heavy preprocessing turned on automatically, which should be turned off for smaller instances, otherwise we might see a slowdown in those cases compared to {\scshape Satch}.
% We also turned off {\scshape Glucose}'s dynamic thread pool option to make its performance more consistent. In general, we see around a $1.5$ to $2$ times speedup from {\scshape Satch} when solving larger instances. 

\subsection{Binary Search}
In order to compute the exact value of a Rado number $R_k(\E)$, we often must determine the satisfiability of $F_n^k(\E)$ for many values of $n$. A convenient property of the formulas $F_n^k(\E)$ is that if $m<n$, then we can obtain the formula $F_m^k(\E)$ simply by deleting all the clauses that contain variables $v_j^i$ with $j>m$. Therefore, once we have a formula $F_u^k(\E)$ that is unsatisfiable, we have an upper bound $R_k(\E) \le u$, and we no longer need to do any solution (negative clause) generation. After obtaining a lower bound $R_k(\E) > \ell$ with a satisfiable formula $F_\ell^k(\E)$, we can compute the exact value of the Rado number $R_k(\E)$ using binary search to jump between $\ell$ and $u$. Our initial guesses for suitable bounds on $R_k(\E)$ were made largely through trial and error. However, even with the poor estimates $10\le R_3(x-y = bz) \le 5000$ for $1\le b \le 15$, it is possible to compute the exact values for all of these numbers in under two hours. 
% This is the idea why we are interested in the first integer $n$ which monochromatic solution(s) is guaranteed. 
% Rather than testing one integer at a time and increase until $n$, where a monochromatic solution is guaranteed, we could give a general bound $l, u$ and use binary search to find the first, namely the smallest integer, that a monochromatic solution is induced. 
% The reasons being:
% \begin{itemize}
%     \item We are searching in $\N$, naturally ordered. 
%     \item Binary search is $\mathcal{O}(\log{}n)$, $n$ is the distance between lower and upper bounds.
%     \item Once the upper bound instance is encoded into Boolean algebra, any smaller instances are just a subset of it. Easy to generate rest of the files.
%     \item Usually we cannot predict the Rado Numbers, so lower bounds and upper bounds are far apart.
% \end{itemize}
\section{Acknowledgements}
This project was supported by National Science Foundation grant DMS-1818969.


\bibliographystyle{abbrv}
\bibliography{Bibliography}

\appendix
% \section{Appendices}
\section{Additional Rado Number calculations}\label{AppendixRadoNumbers}
In this section we give additional bounds and exact values for various Rado numbers. 
\subsection{3-color Rado Numbers}
% $R_3(x+11y = 9z) = 210$

Table \ref{Table_Rado_3Color_additional} gives $R_3(ax+ay=bz)$ for $3\le a\le6$, $11\le b \le 20$. 

\begin{table}[H]
\caption{$R_3(ax+ay = bz)$}
\begin{center}
\begin{tabular}{c|cccccc} \label{Table_Rado_3Color_additional}
\diagbox{$b$}{$a$} &3&4&5 & 6 \\
\hline
11 & 2019 & 847 & 1958 & 1188 \\
12 & $\underline{\infty}$ & \underline{54} & 2400 & \underline{1} \\
13 & $\infty$ & 1710 & 3445 & 1963 \\
14 & $\infty$ & \underline{455} & 3675 & \underline{336} \\
15 & $\underline{\infty}$& 5408 & \underline{54} & \underline{105} \\
16 & $\infty$ &$\underline{\infty}$& 5725 & \underline{432} \\
17 & $\infty$ &$\infty$& 8330 & 4743 \\
18 &$\underline{\infty}$&$\underline{\infty}$& 12069 & \underline{54} \\
19 &$\infty$&$\infty$& 16397 & 6726 \\
20 &$\infty$&$\underline{\infty}$&$\underline{\infty}$& \underline{1025} \\
\end{tabular}
\end{center}
\end{table}
% \begin{table}[H]
% \caption{$R_3(ax+ay = bz)$}
% \begin{tabular}{c|cccccc} \label{Table_Rado_3Color_additional}
% \diagbox{$b$}{$a$} &3&4&5 & 6 \\
% \hline
% 11 & 2019 & 847 & 1958 & 1188 \\
% 12 & $\infty$ & \textit{54} & 2400 & \textit{1} \\
% 13 & $\infty$ & 1710 & 3445 & 1963 \\
% 14 & $\infty$ & \textit{455} & 3675 & \textit{336} \\
% 15 & $\infty$& 5408 & \textit{54} & \textit{105} \\
% 16 & $\infty$ &$\infty$& 5725 & \textit{432} \\
% 17 & $\infty$ &$\infty$& 8330 & 4743 \\
% 18 &$\infty$&$\infty$& 12069 & \textit{54} \\
% 19 &$\infty$&$\infty$& 16397 & 6726 \\
% 20 &$\infty$&$\infty$&$\infty$& \textit{1025} \\
% \end{tabular}
% \end{table}

% \begin{itemize}
%     \item $a=3$: \textit{54}, 1125, 2019, $\infty$, $>15000$
%     \item $a=4$: 585, \textit{105}, 847, \textit{54}, 1710, \textit{455}, 5408 , $\infty$, $>15000$
%     \item $a=5$: 1125, \textit{1}, 1958, 2400, 3445, 3675, \textit{54}, 5725, 8330, 12069, 16397, $\infty$
%     \item $a=6$: \textit{54}, \textit{135}, 1188 \textit{1}, 1963 , \textit{336}, \textit{105} , \textit{432}, 4743 , \textit{432}, 6726, \textit{1125}, \textit{455} ,\textit{2019},
% \end{itemize}
\subsection{4-color Rado Numbers }
Table \ref{Table_Rado_4Color_additional} gives some values for the 4-color Rado numbers $R_4(a(x-y) = bz)$. These numbers are considerably more difficult to compute than $R_3(a(x-y) = bz)$, and it took the solver {\scshape CaDiCaL} \cite{BiereFazekasFleuryHeisinger-SAT-Competition-2020-solvers} up to 20 hours to prove some of the upper bounds. Notably, $R_4(x-y = (m-2)z) = m^4-m^3-m^2-m-1$ for $4\le m \le 6$, which implies $S(4,4) = 171,\ S(5,4) = 469,$ and $S(6,4) = 1037$ by results in \cite{BBGeneralizedSchur} and \cite{MyersThesis}.   

\begin{table}[H]
\caption{$R_4(a(x-y) = bz)$}
\begin{center}
\begin{tabular}{c|ccccc} \label{Table_Rado_4Color_additional}
\diagbox{$b$}{$a$} &1 & 2 &3&4&5\\
\hline
1 & 45 & 56 &81&256&625 \\
2 & 171 &\underline{45}&103&\underline{56}&\\
3 & 469  &$>225$&\underline{45}& &\\
4 &1037  & & &&\\
\end{tabular}
\end{center}
\end{table}

\section{Degree of Regularity Values} 
Tables \ref{Table_dor_1} to \ref{Table_dor_5} give the values of $dor(ax+by =cz)$ for all $a,b,c$ with $1\le a,b,c\le 5$. 

\begin{table}[H]
\begin{minipage}[c]{0.45\textwidth}
\caption{$dor(ax+by = z)$}
\begin{tabular}{c|ccccc} \label{Table_dor_1}
\diagbox{$b$}{$a$}&1&2&3&4&5\\
\hline
1&$\infty$&$\infty$&$\infty$&$\infty$&$\infty$\\
2&$\infty$&$2$&$3$&$2$&$3$\\
3&$\infty$&$3$&$2$&$2$&$2$\\
4&$\infty$&$2$&$2$&$2$&$2$\\
5&$\infty$&$3$&$2$&$2$&$2$\\
\end{tabular}
\end{minipage}
\begin{minipage}[c]{0.45\textwidth}
\caption{$dor(ax+by = 2z)$}
\begin{tabular}{c|ccccc}\label{Table_dor_2}
\diagbox{$b$}{$a$}&1&2&3&4&5\\
\hline
1&$\infty$&$\infty$&$3$&$2$&$3$\\
2&$\infty$&$\infty$&$\infty$&$\infty$&$\infty$\\
3&$3$&$\infty$&$3$&$2$&$3$\\
4&$2$&$\infty$&$2$&$2$&$2$\\
5&$3$&$\infty$&$3$&$2$&$2$\\
\end{tabular}
\end{minipage}
\end{table}

\begin{table}[H]
\begin{minipage}[c]{0.45\textwidth}
\caption{$dor(ax+by = 3z)$}
\begin{tabular}{c|ccccc}\label{Table_dor_3}
\diagbox{$b$}{$a$}&1&2&3&4&5\\
\hline
1&$3$&$\infty$&$\infty$&$3$&$3$\\
2&$\infty$&$3$&$\infty$&$2$&$3$\\
3&$\infty$&$\infty$&$\infty$&$\infty$&$\infty$\\
4&$3$&$2$&$\infty$&$3$&$3$\\
5&$3$&$3$&$\infty$&$3$&$3$\\
\end{tabular}
\end{minipage}
\begin{minipage}[c]{0.45\textwidth}
\caption{$dor(ax+by = 4z)$}
\begin{tabular}{c|ccccc}\label{Table_dor_4}
\diagbox{$b$}{$a$}&1&2&3&4&5\\
\hline
1&$2$&$2$&$\infty$&$\infty$&$3$\\
2&$2$&$\infty$&$2$&$\infty$&$2$\\
3&$\infty$&$2$&$3$&$\infty$&$3$\\
4&$\infty$&$\infty$&$\infty$&$\infty$&$\infty$\\
5&$3$&$2$&$3$&$\infty$&$3$\\
\end{tabular}
\end{minipage}
\end{table}

\begin{table}[H]
\begin{minipage}[c]{0.45\textwidth}
\caption{$dor(ax+by = 5z)$}
\begin{tabular}{c|ccccc}\label{Table_dor_5}
\diagbox{$b$}{$a$}&1&2&3&4&5\\
\hline
1&$2$&$3$&$3$&$\infty$&$\infty$\\
2&$3$&$3$&$\infty$&$2$&$\infty$\\
3&$3$&$\infty$&$3$&$3$&$\infty$\\
4&$\infty$&$2$&$3$&$3$&$\infty$\\
5&$\infty$&$\infty$&$\infty$&$\infty$&$\infty$\\
\end{tabular}
\end{minipage}
\end{table}

\section{Heuristic search procedure in proof of Theorem \ref{Theorem3ColorRadoFamilies}}
Here we detail the method {\scshape FindPolynomials} which we used to find the sets $S$ of polynomials in the proof of Theorem $\ref{Theorem3ColorRadoFamilies}$. We give the version of {\scshape FindPolynomials} used for the equation $a(x-y) = (a-1)z$. The procedures for the other two equations were similar, but had minor modifications. 

In brief, we initialize $S$ to a set of polynomials $S_0$, and we use an auxiliary set of ``gaps" $G$ to add more polynomials to $S$. The procedure {\scshape BoundedIntegerPolynomial} returns true if and only if all of its arguments are polynomials $p(a) \in \Z[a]$ that satisfy $1\le p(a) \le a^3+(a-1)^2$ for all $a \ge 16$. The {\scshape FindPolynomials} procedure is not guaranteed to produce a set of clauses that yields an unsatisfiable formula, and it took several attempts to come up with suitable choices for the initial sets $S_0$ and $G_0$. For the equation $a(x-y) = (a-1)z$, we set $S_0 = \{1,a-1,a,a+1,a^2-1,a^2,a^2+1,a^3,a^3+(a-1)^2\}$, $G_0 = \{1,a-1,a,(a-1)^2,a^2\}$, and maxIterations = 3. Files containing the polynomials and clauses used in our formulas can be found in \cite{Ramsey_Research_Software}.
\begin{algorithm}[H]\label{SearchAlgorithm}
\caption{{\scshape FindPolynomials}($S_0,G_0$,maxIterations)}
\begin{algorithmic}

\State{$S \gets S_0$}
\State{$G \gets G_0$}
\For{$i = 1$ to maxIterations}
\For{$p,q$ in $S$}
 \State{$r \gets \frac{p-q}{a-1}$}
 \If{{\scshape BoundedIntegerPolynomial}($r$)}
 \State{$G \gets G \cup \{r\}$}
 \EndIf
\EndFor
\For{$p \in S$, $q \in G$}
\State{$r_+ \gets p+(a-1)q$}
\If{{\scshape BoundedIntegerPolynomial}($r_+$)}
\State{$S \gets S \cup \{r_+\}$}
 \EndIf
 \State{$r_- \gets p-(a-1)q$}
\If{{\scshape BoundedIntegerPolynomial}($r_-$)}
\State{$S \gets S \cup \{r_-\}$}
 \EndIf
\EndFor
\EndFor
\State{$C \gets \emptyset$}
\For{$p,q \in S$}
\State{$x \gets p$}
\State{$y \gets p-(a-1)q$}
\State{$z \gets aq$} \Comment{Here $(x,y,z)$ is a solution to $a(x-y) = (a-1)z.$}
\If{{\scshape BoundedIntegerPolynomial}($x,y,z$)}
 \State{$S \gets S \cup \{x,y,z\}$}
 \State{$C \gets C \cup \{\{x,y,z\}\}$}
\EndIf
\EndFor
\State{\Return{$S,C$}}

\end{algorithmic}
\end{algorithm}

\end{document}